\documentclass[11pt]{amsart}

\usepackage{amsmath, amsthm, amsfonts, amssymb}
\newtheorem{prop}{Proposition}
\newtheorem{theorem}{Theorem}
\newtheorem{coroll}{Corollary}
\newtheorem{lemma}{Lemma}

\newcommand{\BRW}{\mbox{BRW}}
\newcommand{\ig}{[\hspace{-1.5pt}[}
\newcommand{\id}{]\hspace{-1.5pt}]}
\def\floor#1{\lfloor #1 \rfloor}
\def\ceil#1{\lceil #1 \rceil}
\newcommand{\Z}{\mathbb{Z}}
\newcommand{\N}{\mathbb{N}}

\newcommand{\C}{\mathcal{C}}

\newcommand{\un}{\mathbf{1}}
\renewcommand{\P}{\mathbb{P}}

\newcommand{\E}{\mathbb{E}}
\newcommand{\R}{\mathbb{R}}

\newcommand{\F}{\mathcal{F}}
\newcommand{\Q}{\mathbb{Q}}

\newcommand{\T}{\mathcal{T}}
\newcommand{\TT}{\mathfrak{T}}

\author{Jean B\'{e}rard,  Jean-Baptiste Gou\'{e}r\'{e}}
\address[Jean B\'{e}rard]{\noindent Institut Camille Jordan, UMR CNRS 5208, 43, boulevard du 11 novembre
1918, Villeurbanne, F-69622, France; universit\'{e} de Lyon, Lyon, F-69003, France; 
universit\'{e} Lyon 1, Lyon, F-69003, France
\newline
e-mail:  \rm \texttt{jean.berard@univ-lyon1.fr}}
\address[Jean-Baptiste Gou\'{e}r\'{e}]{\noindent Laboratoire MAPMO -
  UMR 6628, Universit\'e d'Or\-l\'eans, B.P. 6759, 45067 Orl\'eans Cedex
  2, France. 
\newline 
E-mail: \rm
  \texttt{Jean-Baptiste.Gouere@univ-orleans.fr}.}

\date{}

\title[Brunet-Derrida behavior of particle systems]{Brunet-Derrida behavior of branching-selection particle systems on the line}

\begin{document}

\begin{abstract}
We consider a class of branching-selection particle systems on $\R$ similar to the one considered by E. Brunet and B. Derrida in their 
1997 paper "Shift in the velocity of a front due to a cutoff". Based on numerical simulations and heuristic arguments, 
Brunet and Derrida showed that, as the population size $N$ of the particle system goes to infinity, the asymptotic velocity of the system converges to 
a limiting value at the unexpectedly slow rate $(\log N)^{-2}$. In this paper, we give a rigorous mathematical proof of this fact, for the class of particle systems we consider.
The proof makes use of ideas and results by R. Pemantle, and by N. Gantert, Y. Hu and Z. Shi, and relies on a comparison of the particle system with a family of $N$ independent
branching random walks killed below a linear space-time barrier. 
\end{abstract}

\maketitle

\section{Introduction}

\subsection{Brunet-Derrida behavior}

In~\cite{BruDer1, BruDer2}, E. Brunet and B. Derrida studied, among other things, a discrete-time particle system on $\Z$, 
in which a population of particles with fixed size $N$ undergoes repeated steps of branching and selection. As time goes to infinity, the population of $N$
particles, taken as a whole, moves ballistically, with an asymptotic speed depending on the population size $N$. One remarkable property of this system is the following: 
as $N$ goes to infinity, the asymptotic speed of the population of particles converges to a limiting value,  
but at the unexpectedly slow rate of $(\log N)^{-2}$, bringing to light an unusually large finite-size effect. This behavior was, on the one hand, observed 
by Brunet and Derrida on direct numerical simulations of the particle system (with large numbers of particles, up to $N=10^{16}$).
On the other hand, Brunet and Derrida provided a justification for this behavior through the following argument. First, in the limit where $N$ goes to infinity, 
the time-evolution of the distribution of particles in the branching-selection system is governed by a deterministic equation, which can be viewed as a discrete version of the 
well-known F-KPP equation
\begin{equation}\label{e:F-KPP}\frac{\partial u}{\partial t}  =  \Delta u + u(1-u),\end{equation} 
where $u=u(x,t)$, $x \in \R$,  $t \geq 0$. 
 To account for the fact that there is only a finite number $N$ of particles in the system instead of an infinite one -- whence a resolution equal to $1/N$ for representing distributions of mass --, 
 one may introduce a cut-off value of $1/N$ in the equation, and expect that this modified equation still reflects at least some of the behavior of the original particle system. 
 Whence the question of studying, for large $N$, an equation of the form:
\begin{equation}\label{e:F-KPP-cutoff}\frac{\partial u}{\partial t}  =  \Delta u + u(1-u) \un( u \geq 1/N).\end{equation} 
In fact, Brunet and Derrida could provide heuristic arguments for this new problem, showing that, for large $N$, 
the effect of the cut-off is to shift the speed of travelling wave solutions of Equation~(\ref{e:F-KPP-cutoff}) from the speed of those of Equation~(\ref{e:F-KPP}), 
by an amount of order $(\log N)^{-2}$. In turn, these arguments were supported by numerical simulations of (discrete versions of) Equation~(\ref{e:F-KPP-cutoff}). 
This result concerning the F-KPP equation with cut-off has recently been given a rigorous mathematical proof, see~\cite{BenDep2, BenDep1, DumPopKap}. 

A related question (see~\cite{BruDer3}), is that of the behavior of the F-KPP equation with small noise, i.e. of the equation
\begin{equation}\label{e:F-KPP-noise}\frac{\partial u}{\partial t}  =  \Delta u + u(1-u)+\sqrt{\frac{u(1-u)}{N}} \dot{W},\end{equation} 
where $\dot{W}$ is a standard space-time white-noise,  and $N$ is large. 
Rigorous results have recently been derived for this model too, see~\cite{ConDoe, MueMytQua, MueMytQua2}, establishing that the speed of the random 
travelling wave solutions of Equation~(\ref{e:F-KPP-noise}) is, for large $N$, shifted from the speed of those of~(\ref{e:F-KPP}) by an amount of order $(\log N)^{-2}$.

We thus have (at least) three examples of what may be called Brunet-Derrida behavior, in three different and more or less loosely related frameworks (branching-selection particle systems, F-KPP-equation with cut-off, F-KPP equation with noise), 
two of which have already been established rigorously.  

\subsection{Main result}\label{ss:main}

The goal of this paper is to give a proof of Brunet-Derrida behavior for a class of branching-selection systems that is similar  (but not exactly identical) to the one originally studied 
by Brunet and Derrida in~\cite{BruDer1, BruDer2}.
 
 To be specific, we consider a discrete-time particle system with $N$ particles on $\R$ evolving through the repeated application of branching and selection steps defined as follows:
\begin{itemize}
\item Branching: each of the $N$ particles is replaced by two new particles, whose positions are shifted from that of the original particle by independently
 performing two random walk steps, according to a given distribution $p$;
\item Selection:  only the $N$ rightmost particles are kept among the $2N$ obtained at the branching step, to form the new population of $N$ particles.
\end{itemize}

Our assumptions on the random walk distribution $p$ are listed below, and come from the need to apply the result of the paper~\cite{GanHuShi} by N. Gantert, Y. Hu and Z. Shi on 
the survival probability of the branching random-walk killed below a linear space-time boundary, in the special case of deterministic binary branching.

Introduce the logarithmic moment generating function of $p$ defined by 
$$\Lambda(t):= \log \int \exp(tx) dp(x).$$ 
Here are the assumptions on $p$:
\begin{itemize}
\item[(A1)]  The number $\sigma :=  \sup \{ t \geq 0; \ \Lambda(-t)<+\infty \}$ is $>0$.
\item[(A2)]  The number $\zeta:= \sup \{ t \geq 0; \ \Lambda(t)<+\infty \}$ is $>0$.
\item[(A3)]  There exists $t^{*} \in ]0,\zeta[$ such that $t^{*} \Lambda'(t^{*}) - \Lambda(t^{*}) = \log 2$.
\end{itemize}
Under these assumptions, both numbers $$\chi(p):=\textstyle{\frac{\pi^{2}}{2}} t^{*} \Lambda''(t^{*}), \  v(p):=\Lambda'(t^{*})$$
are well-defined, and satisfy $0<\chi(p)<+\infty$ and  $v(p) \in \R$. 
Simple cases for which these assumptions hold are e.g. the Bernoulli case  for $\alpha \in ]0,1/2[$, where 
$p=\alpha \delta_{1} + (1-\alpha) \delta_{0}$, the uniform case, where $p$ is the uniform distribution on the interval $[0,1]$, and the gaussian case, where $p$ is the standard 
Gaussian distribution on $\R$. 

In Section~\ref{s:elementary} below, it is proved that, after a large number of iterated branching-selection steps, 
the displacement of the whole population of $N$ particles is ballistic, with deterministic asymptotic speed $v_{N}(p)$, 
and that, as $N$ goes to infinity, $v_{N}(p)$ increases to a limit $v_{\infty}(p)$, which  turns out to be equal to the $v(p)$ defined above, and is thus finite under our assumptions.
The main result concerning the branching-selection particle system is the following theorem:
\begin{theorem}\label{t:BD}
Assume that (A1)-(A2)-(A3) hold. Then, as $N$ goes to infinity,
\begin{equation}\label{e:BD}  v_{\infty}(p) - v_{N}(p) \sim \chi(p) (\log N)^{-2}.\end{equation}
\end{theorem}

\subsection{Credits}

The proof of Theorem~\ref{t:BD} given in this paper is based on a comparison of the particle system with a family of $N$ independent
branching random walks killed below a linear space-time barrier, and makes use in a crucial way of ideas and results from the following two sources: 
the paper~\cite{Pem} by R. Pemantle on complexity bounds for algorithms seeking near optimal paths in branching
random walks, and the paper~\cite{GanHuShi} by Gantert, Hu and Shi on the survival probability of the branching random-walk killed below a linear space-time boundary. 
A detailed description of exactly which ideas and results are used and how is given in Sections~\ref{s:BRW}, \ref{s:lower} and~\ref{s:upper} below. 
Note that the existence of a link between the 
Brunet-Derrida behavior of a branching-selection particle system such as the
 one studied here, and the asymptotics of the survival probability for branching random walks killed below a linear space-time barrier, was already suggested in the 
 papers~\cite{DerSim1, DerSim2} by B. Derrida and D. Simon, where Brunet-Derrida-like features were observed for a quasi-stationary regime of killed branching random walks; 
the present paper gives an explicit and rigorous version of such a relation.

Finally, let us mention that a first version~\cite{Ber} of the present work was completed by one of the authors (J.B.) before the results in~\cite{GanHuShi} became publicly available. 
In~\cite{Ber}, only the $(\log N)^{-2}$ order of magnitude of the  difference $v_{\infty}(p) - v_{N}(p)$ in the Bernoulli case was established. 
The results in~\cite{GanHuShi} then allowed us to prove Theorem~\ref{t:BD}, which is both more precise and more general. 

\subsection{Organization of the paper}
 
The rest of the paper is organized as follows. In Section~\ref{s:not-def}, we provide the precise notations and definitions that are needed 
in the sequel. Section~\ref{s:elementary} contains a discussion of various elementary properties of the model we consider. Section~\ref{s:BRW} collects the results from~\cite{GanHuShi} that
are used in the sequel.
Section~\ref{s:lower} contains the proof of the lower bound part of Theorem~\ref{t:BD}, while Section~\ref{s:upper} contains the proof of the upper bound part.
Section~\ref{s:autres-cas} discusses the Bernoulli($\alpha$) case for $\alpha \geq 1/2$, showing that the conclusion of Theorem~\ref{t:BD} may fail to hold
when Assumption (A3) is not met. 
Section~\ref{s:discussion} is an attempt to provide a self-contained explanation of the $(\log N)^{-2}$ order of magnitude 
appearing in Theorem~\ref{t:BD}. The arguments in this section are only discussed in an informal way. 

\section{Notations and definitions}\label{s:not-def}

\subsection{Particle systems on $\R$}

It is convenient to represent finite populations of particles by finite counting measures on 
$\R$. We use the notation $\C$ to represent the set of all finite counting measures on $\R$.

For $\nu \in \C$, the total mass of $\nu$ (i.e. the number of particles in the population it describes) is denoted by $M(\nu)$. We 
denote by $\max \nu$ and $\min \nu$ respectively the maximum and minimum of the (finite) support of $\nu$.  
We also define the diameter $d(\nu):=\max \nu - \min \nu$. 
Given $\mu, \nu \in \C$, we use the notation $\prec$ to denote the usual stochastic ordering: $\mu \prec \nu$ if and only if
$\mu([x,+\infty[) \leq \nu([x,+\infty[)$ for all $x \in \R$. In particular, $\mu \prec \nu$ implies that $M(\mu) \leq M(\nu)$, and it is easily seen that, 
if $\mu=\sum_{i=1}^{M(\mu)} \delta_{x_{i}}$ and $\nu=\sum_{i=1}^{M(\nu)} \delta_{y_{i}}$, with $x_{1}\geq \cdots \geq x_{M(\mu)}$ and
$y_{1} \geq \cdots \geq y_{M(\nu)}$, $\mu \prec \nu$ is equivalent to $M(\mu) \leq M(\nu)$ and $x_{i} \leq y_{i}$ for all $i \in \ig  1, M(\mu)\id$.

For all $N \geq 1$, let $\C_{N}$ denote the set of finite counting measures on $\R$ with total mass equal to $N$.
In the sequel, we use the notation $(X^{N}_{n})_{n \geq 0}$ to denote a Markov chain on $\C_{N}$ whose transition probabilities are given by 
the branching-selection mechanism with $N$ particles defined in Section~\ref{ss:main}, 
and which starts at a deterministic value $X^{N}_{0} \in \C_{N}$. 
We assume that this Markov chain is defined on a reference probability space denoted by $(\Omega,\F,\P)$.

\subsection{Branching random walks}

In the sequel, we use the notation $\BRW$ to denote a generic branching random walk on a regular rooted binary tree, with value zero 
at the root, and i.i.d. displacements with common distribution $p$ along each edge. More formally, $\BRW$ consists of a pair
$(\TT, \Phi)$, where $\TT$ is a regular rooted binary tree, and $\Phi$ is a random map, associating to each vertex $u \in \TT$ a random variable $\Phi(u) \in \R$
in such a way that $\Phi(root)=0$ and that the collection $\left(\Phi(v)-\Phi(u) \right)_{(u,v)}$ is i.i.d. with common distribution $p$, where $(u,v)$ runs over the set of pairs of vertices
of $\TT$ such that $u$ is the father of $v$. We say that $\Phi(u)$ is the value of the branching random walk at vertex $u$. The probability measure governing $\BRW$ 
 is denoted by $\Q$.

Given $m \geq 1$, we say that a sequence $u_{0},\ldots, u_{m}$ of vertices in $\TT$ is a descending path if, for all $i \in \ig 1,  m\id$, $u_{i-1}$ is the father of $u_{i}$. 
The set of vertices of $\TT$ located at depth $m$ is denoted by $\TT(m)$.

\section{Elementary properties of the model}\label{s:elementary}

As a first quite elementary property, note that, from Assumptions (A1) and (A2), $\E(\max X^{N}_{n})$ and   $\E(\min X^{N}_{n})$ are finite for all $n \geq 0$, 
for any choice of the (deterministic) initial condition $X^{N}_{0} \in \C_{N}$.

\subsection{Estimates on the diameter}

\begin{prop}\label{p:diam-borne}
Let $u_{N}:=\ceil{\frac{\log N}{\log 2}}+1$.  For all $N \geq 1$, all initial population $X^{N}_{0} \in \C_{N}$, and all $n \geq u_{N}$, $ d(X^{N}_{n})$ is stochastically dominated by 
$u_{N} \times (m_N^{(2)}-m_N^{(1)})$, where $m_N^{(2)}$ and  $m_N^{(1)}$ are respectively the maximum and the minimum of a family of $2Nu_{N}$ i.i.d. random variables with common distribution $p$. 
 \end{prop}

\begin{proof}
Consider $n \geq  u_{N}$, $y=\max X^{N}_{n-u_{N}}$, and let us study the evolution of the branching-selection system between times $n-u_N$ and $n$.
Define $m_N^{(1)}$ as the minimum of the $2Nu_N$ random walk steps 
performed by the system between times $n-u_N$ and $n$.

Consider first the possibility that $\min X^{N}_{k} < y + (k-n+u_N) m_N^{(1)}$ for all $k \in \ig n+1-u_{N} , n \id$. Since all the random walk steps that are performed during branching steps are $\geq m_N^{(1)}$, this implies
that all the particles descended by branching from a particle located at $y$ at time $n-u_{N}$, are preserved by the successive selection steps performed from $X^{N}_{n-u_{N}}$ to $X^{N}_{n}$.
Since there are at least $2^{u_{N}}>N$ such particles at time $n$, this is a contradiction. As a consequence, we know that there must be an index  $k \in \ig n+1-u_{N} , n \id$ such that 
$\min X^{N}_{k} \geq y + (k-n+u_N) m_N^{(1)}$. Again by the fact that random walk steps are $\geq  m_N^{(1)}$,  $t \mapsto \min X^{N}_{t}-  (t-n+u_N) m_N^{(1)}$ is non-decreasing on the interval  $\ig n+1-u_{N} , n \id$, so we deduce that $\min X^{N}_{n} \geq y+u_N  m_N^{(1)}$. 
Now, let $m_N^{(2)}$ denote the maximum of the $2N u_{N}$ random walk steps that are performed at the branching steps between time $n-u_{N}$ and time $n$. 
We see from the definition of $y$ that $\max X^{N}_{n} \leq y+u_{N} m_N^{(2)}$.  We have also just seen that  $\min X^{N}_{n} \geq y + u_N  m_N^{(1)}$, so that
 $d(X^{N}_{n})=\max X^{N}_{n}-\min X^{N}_{n} \leq u_{N}  (m_N^{(2)}- m_N^{(1)})$.

 \end{proof}

The following corollary is then a rather straightforward consequence, in view of Assumptions (A1) and (A2).
\begin{coroll}\label{c:diam-negligeable}
For all $N \geq 1$ and any initial population $X^{N}_{0} \in \C_{N}$, $\lim_{n \to +\infty}  n^{-1} d(X^{N}_{n})=0$, both with probability one and in $L^{1}(\P)$.
\end{coroll}

\begin{proof}
Using the notations of Proposition \ref{p:diam-borne}, let $F_N := m_N^{(2)}-m_N^{(1)}$.
From Assumptions (A1) and (A2), one deduces that $E(F_N)<+\infty$. Then, by Proposition  
\ref{p:diam-borne}, one has that, for all $n \geq u_N$, 
$\E(n^{-1}d(X^{N}_{n})) \leq E(F_N)/n$, so that convergence to $0$ in $L^{1}(\P)$ is proved.
Moreover, for any $\iota>0$, Proposition \ref{p:diam-borne} yields that $\sum_{n \geq u_N} \P(n^{-1}d(X^{N}_{n}) \geq \iota) \leq \sum_{n \geq u_N} P(F_N \geq \iota n  ) \leq E(F_N)/\iota$, so that convergence to $0$ $\P-$a.s. follows from the Borel-Cantelli lemma, since $\iota$ can be taken arbitrarily small.
\end{proof}

\subsection{Monotonicity properties}

The following lemma states a key monotonicity property of our branching-selection mechanism.

\begin{lemma}\label{l:kernel-monot}
For all $1 \leq N_{1} \leq N_{2}$, and $\mu_{1} \in \C_{N_{1}}$, $\mu_{2} \in \C_{N_{2}}$ such that $\mu_{1} \prec \mu_{2}$, there exists a pair of random variables
$(Z^{1},Z^{2})$ taking values in $\C_{N_{1}} \times \C_{N_{2}}$, such that:
\begin{itemize}
\item  the distribution of 
$Z^{i}$ for $i=1,2$ is that of the population of particles obtained by performing one branching-selection step (with $N_{i}$ particles) starting from the population $\mu_{i}$;
\item with probability one, $Z^{1} \prec Z^{2}$. 
\end{itemize}
\end{lemma}

\begin{proof}
Consider an i.i.d. family $(\varepsilon_{i,j})_{i \in \ig 1, N_{2} \id, \ j=1,2}$ with common distribution $p$. 
For $k=1,2$, write $\mu_{k}=\sum_{i=1}^{N_{k}} \delta_{x_{i}(k)}$, with
$x_{1}(k) \geq \ldots \geq x_{N_{k}}(k)$.
Then let $T_{k} := \sum_{i=1}^{N_{k}} \sum_{j=1,2} \delta_{x_{i}(k)+\varepsilon_{i,j}}$, and define $Z^{k}$ as being formed by the $N_{k}$
rightmost particles in $T_{k}$.
From the assumption that $\mu_{1} \prec \mu_{2}$, we deduce that
 $x_{i}(1) \leq x_{i}(2)$ for all $i \in \ig 1, N_{1}\id$, whence the fact that 
 $x_{i}(1)+\varepsilon_{i,j} \leq  x_{i}(2)+\varepsilon_{i,j}$, for all $1 \leq i \leq N_{1}$ and $j=1,2$.
 It is easy to deduce that $T_{1} \prec T_{2}$, whence $Z^{1} \prec Z^{2}$. 
The conclusion follows.
\end{proof}

An immediate corollary is the following.

\begin{coroll}\label{c:kernel-monot}
For all $1 \leq N_{1} \leq N_{2}$, and $\mu_{1} \in \C_{N_{1}}$, $\mu_{2} \in \C_{N_{2}}$ such that $\mu_{1} \prec \mu_{2}$, 
there exists a coupling $(Z^{1}_{n}, Z^{2}_{n})_{n \geq 0}$ between two versions of the branching-selection particle system, with $N_{1}$ and $N_{2}$ particles respectively, 
such that $Z^{1}_{0}:=\mu_{1}$,   $Z^{2}_{0}:=\mu_{2}$, and $Z^{1}_{n} \prec Z^{2}_{n}$ for all $n \geq 0$.
\end{coroll}

\begin{prop}\label{p:subadd}
There exists $v_{N}(p) \in \R$ such that, with probability one, and in $L^{1}(\P)$, for any initial population $X^{N}_{0} \in \C_{N}$,
$$\lim_{n \to +\infty} n^{-1} \min X^{N}_{n} =  \lim_{n \to +\infty} n^{-1} \max X^{N}_{n} = v_{N}(p).$$ 
\end{prop}

\begin{proof}
Note that, in view of Corollary~\ref{c:diam-negligeable},  if either of the two limits in the above statement exists, then the other must exist too and have the same value. 
Moreover, owing to the translation invariance
of our particle system (the dynamics is invariant with respect to shifting all the particles by a translation on $\R$), and to Corollary~\ref{c:kernel-monot}, we see that it is enough to prove the result when $X^{N}_{0} = N \delta_{0}$. The idea of the proof is to invoke 
Kingman's subadditive ergodic theorem (see e.g.~\cite{Dur}), using the monotonicity property described by Lemma~\ref{l:kernel-monot}.

Consider an i.i.d. family $(\varepsilon_{\ell,i,j})_{\ell \geq 0, i \in \ig 1, N \id, \ j=1,2}$ with common distribution $p$ (the index $\ell$ will be used to shift the origin of time when applying Kingman's theorem). 
For all $\ell \geq 0$, denote by $(W^{N}_{\ell,k})_{k \geq 0}$ the branching-selection system starting at 
$W^{N}_{\ell,0}:=N \delta_{0}$ and governed by the following steps. 
For $k \geq 0$,  write $W^{N}_{\ell,k}=\sum_{i=1}^{N} \delta_{x_{i}}$, with
$x_{1} \geq \ldots \geq x_{N}$. The population $T_{\ell,k}$ derived from $W^{N}_{\ell,k}$ by branching is then defined by 
$T_{\ell,k}:=\sum_{(i,j) \in \ig 1,N \id \times \{ 1,2 \}} \delta_{x_{i}+\varepsilon_{\ell+k,i,j}}$. Then, $W^{N}_{\ell,k+1}$ is obtained 
from $T_{\ell,k}$ by keeping only the $N$ rightmost particles.

Observe that $(W^{N}_{0,n})_{n \geq 0}$ has the same distribution as $(X^{N}_{n})_{n \geq 0}$.
Moreover, the argument used in the proof of Lemma~\ref{l:kernel-monot} shows that 
\begin{equation}\label{e:sous-additivite}   \mbox{ for all $n,m\geq 0$, }\max W^{N}_{0,n+m} \leq \max W^{N}_{0,n}+\max W^{N}_{n,m}.\end{equation}
Indeed, it is enough to note that~(\ref{e:sous-additivite}) compares the maximum of two populations obtained by performing $m$ branching-selection steps coupled as in Lemma~\ref{l:kernel-monot}, 
starting respectively from $W^{N}_{0,n}$ (for the l.h.s.) and from $N \delta_{\max W^{N}_{0,n}}$ (for the r.h.s.). 

Moreover, it is easily seen from the definition that, for each $d \geq 1$, the random variables $(W^{N}_{dn, d})_{n \geq 0}$ form an i.i.d. family, and that the distribution of $(W^{N}_{\ell,k})_{k \geq 0}$ clearly does not depend on $\ell$.  
One can then check from the definition that the following inequality holds $ \left| \max W^{N}_{0,n} \right|\leq \sum_{k=0}^{n-1} \sum_{(i,j) \in \ig 1,N \id \times \{ 1,2 \}} |\varepsilon_{\ell+k,i,j}|.$
Using assumptions (A1) and (A2), it is then quite clear that there exists $\psi>-\infty$ such that,  for all $n \geq 0$,  $ \psi n \leq E(W^N_{0,n}) < +\infty$.

We conclude that the hypotheses of Kingman's subadditive ergodic theorem hold (see e.g.~\cite{Dur}), and deduce that
 $\lim_{n \to +\infty} n^{-1} \max X^{N}_{n}$ exists both a.s. and in $L^{1}(\P)$, and is constant. 
\end{proof}

\begin{prop}
The sequence $(v_{N}(p))_{N \geq 1}$ is non-decreasing.
\end{prop}

\begin{proof}
Consequence of the fact that, when $N_{1} \leq N_{2}$, $N_{1} \delta_{0} \prec N_{2} \delta_{0}$, and of the monotonic coupling property given in Corollary~\ref{c:kernel-monot}.
\end{proof}
We can deduce from the above proposition that there exists $v_{\infty}(p)$ such that
$\lim_{N \to +\infty} v_{N}(p) = v_{\infty}(p)$. A consequence of the proof of 
Theorem~\ref{t:BD} below is that 
$v_{\infty}(p)$ is in fact equal to the number $v(p):=\Lambda'(t^{*})$, which is finite from our assumptions on $p$.

\subsection{Coupling with a family of $N$ branching random walks}\label{ss:cwafonbrw}

Let  $(\BRW_{i})_{i \in \ig1,N \id}$, denote $N$ independent copies of a branching random walk $\BRW$ as defined in Section~\ref{s:not-def}.
Each $\BRW_{i}$ thus consists of a binary tree $\TT_{i}$ and a map $\Phi_{i}$.
For $1 \leq i \leq N$, and $n \geq 0$, remember that $\TT_{i}(n)$ denotes the set of vertices of $\TT_{i}$ located at depth $n$, 
and define the disjoint union $\T^{N}_{n} := \TT_{1}(n) \sqcup \cdots \sqcup \TT_{N}(n)$.
 For every $n$, fix an a priori 
(i.e. depending only on the tree structure, not on the random walk values) total order on $\T^{N}_{n}$.
We now define by induction a sequence $(G^{N}_{n})_{n \geq 0}$ such that, for each $n \geq 0$, $G^{N}_{n}$ is a random subset of  
$\T^{N}_{n}$ with exactly $N$ elements. 
 First, set $G^{N}_{0}:= \T^{N}_{0}$. 
Then, given $n \geq 0$ and $G^{N}_{n}$, let $H^{N}_{n}$ denote the subset of  $\T^{N}_{n+1}$ formed by the children (each with respect to the tree structure it belongs to) 
of the vertices in $G^{N}_{n}$. Then, define $G^{N}_{n+1}$ as the subset of $H^{N}_{n}$ formed by the $N$ vertices that are associated 
with the largest values of the underlying random walks $\Phi_{i}$s (breaking ties by using the a priori order on $\T^{N}_{n}$). 
Now let $\mathfrak{X}^{N}_{n}$ denote the (random) empirical distribution describing the values taken by the $\Phi_{i}$s
on the (random) set of vertices $G^{N}_{n}$. 
The sequence $(\mathfrak{X}^{N}_{n})_{n \geq 0}$ has the same distribution as $(X^{N}_{n})_{n \geq 0}$, when started from $X^{N}_{0}:=N \delta_{0}$. 
  Thus, we can take for our reference probability space $(\Omega, \F, \P)$ the one on which  $\BRW_{1},\ldots, \BRW_{N}$ are defined, 
 and let $X^{N}_{n}$ be equal to the empirical distribution associated with the subset $G^{N}_{n}$, and so obtain a coupling 
 between $(X^{N}_{n})_{n \geq 0}$ (with $X^{N}_{n}= N \delta_{0}$) and the $N$ branching random walks $\BRW_{1},\ldots, \BRW_{N}$.

\section{Results on the branching random walk killed below a linear space-time barrier}\label{s:BRW}

Let us start with the following definition, adapted from~\cite{Pem}. 
Given $v \in \R$ and $m \geq 1$, we say that a vertex $u \in \BRW$ is $(m,v)-$good if there exists a finite descending path $u=:u_{0},u_{1},\ldots, u_{m}$ such that 
$\Phi(u_{i}) -\Phi(u_{0}) \geq v i$ for all $i \in \ig 0,m\id$. Similarly, 
we say that $u$ is $(\infty,v)-$good if  there exists an infinite descending path $u=:u_{0},u_{1},\ldots$ such that 
$\Phi(u_{i}) -\Phi(u_{0}) \geq v i$ for all $i \in \ig 0,+\infty\ig$.

With this terminology, the main result in~\cite{GanHuShi} can be stated as follows, remembering that 
$v(p)=\Lambda'(t^{*})$ and $\chi(p)=\textstyle{\frac{\pi^{2}}{2}} t^{*} \Lambda''(t^{*})$.
\begin{theorem}\label{t:borne-bien-original}(Theorem 1.2 in~\cite{GanHuShi})
Let $\rho(\infty,\epsilon)$ denote the probability that the root of $\BRW$ is $(\infty,v(p)-\epsilon)-$good.  
Then, as $\epsilon$ goes to zero,  
$$\rho(\infty,\epsilon)=\exp\left(-\left[\frac{\chi(p)+o(1)}{\epsilon} \right]^{1/2}\right).$$
\end{theorem}

We shall need a result which, although not stated explicitly in~\cite{GanHuShi}, appears there as an intermediate step in a proof.
\begin{theorem}\label{t:borne-bien}(Proof of  the upper bound part of  Theorem 1.2 in~\cite{GanHuShi})
Let $\rho(m,\epsilon)$ denote  the probability that the root of $\BRW$  is $(m,v(p)-\epsilon)-$good.  For any $0<\beta<\chi(p)$, there exists $\theta>0$ such that, for all large $m$,
 $$\rho(m,\epsilon) \leq  \exp\left(-\left[\frac{\chi(p)-\beta}{\epsilon} \right]^{1/2}\right), \mbox{ with } \epsilon:=\theta/m^{2/3}.$$
\end{theorem}

One should also consult the papers~\cite{DerSim1, DerSim2} for an approach of these results based on (mathematically non-rigorous) theoretical physics arguments. See also the discussion in Section~\ref{s:discussion}.

\section{The lower bound}\label{s:lower}

The arguments used here in the proof of the lower bound, combine ideas from the paper~\cite{Pem} by Pemantle, which deals with the closely related question 
of obtaining complexity bounds for algorithms that seek near optimal paths in branching random walks, and the estimate on~$\rho(m,\epsilon)$ 
from the paper~\cite{GanHuShi}, by Gantert, Hu and Shi. In fact, the proof given below is basically a rewriting of the proof of 
the lower complexity bound in~\cite{Pem} in the special case of algorithms that do not jump, with the following slight differences:  we are dealing with $N$ independent branching 
random walks being explored in parallel, rather than with a single branching random walk; we consider possibly unbounded
random walk steps; we use the estimate in~\cite{GanHuShi} instead of the cruder one derived in~\cite{Pem}.

We start with an elementary result adapted from~\cite{Pem}.
\begin{lemma}\label{l:beaucoup-de-bien}(Adapted from Lemma 5.2 in~\cite{Pem}.)
Let $v_1,v_2 \in \R$ be such that $v_{1}<v_{2}$, $n \geq 1$, $m \in \ig 1,n\id$, $K > 0$, and let $0=:x_{0},\ldots, x_{n}$ be a sequence of real numbers 
such that $x_{i+1}-x_{i} \leq K$ for all $i  \in \ig 0, n-1 \id$.
Let $I := \{  i \in \ig 0, n-m \id ; \   x_{j} - x_{i}\geq v_{1} (j-i) \mbox{ for all }j \in \ig i,i+m\id \}$.
If $ x_{n} \geq v_{2} n$, then  $ \# I \geq \textstyle{ \frac{v_{2}-v_{1}}{K-v_{1}}}\frac{n}{m} -K/(K-v_{1})$.
\end{lemma}

Since Lemma~\ref{l:beaucoup-de-bien} admits so short a proof, we give it below for the sake of completeness, even though it is quite similar to that in~\cite{Pem}. 

\begin{proof}[Proof of Lemma~\ref{l:beaucoup-de-bien}](Adapted from~\cite{Pem}.)
Consider a sequence $0=:x_{0},\ldots, x_{n}$ as in the statement of the lemma.
 Let then $\tau_{0}:=0$, and, given $\tau_{i} \leq n$, define inductively $\tau_{i+1}:=\inf \{ j \in \ig \tau_{i}+1, n  \id ; \  x_{j} < x_{\tau_{i}}+v_{1} (j-\tau_{i})
 \mbox{ or } j=\tau_{i}+m  \}$, with the convention that $\inf \emptyset = n+1$. Now "color" the integers $k \in \ig 0 , n-1 \id$, 
 according to the following rules: if $x_{\tau_{i+1}} \geq x_{\tau_{i}}+v_{1} (\tau_{i+1}-\tau_{i})$ and $\tau_{i+1} \leq n$, then 
 $\tau_{i},\ldots,\tau_{i+1}-1$ are colored 
 red. Note that this yields a segment of $m$ consecutive red terms, and that $\tau_{i}$ then belongs to $I$. 
 Then color in blue the remaining integers in $\ig 0, n-1 \id$.

  Let $V_{red}$ (resp. $V_{blue}$) denote the number of red (resp. blue) terms
 in $\ig 0 , n-1 \id$. 
 Then decompose the value of $x_{n}$ into the contributions of the steps $x_{k+1}-x_{k}$ such that $k$ is red, and such that $k$ is blue, respectively. 
 On the one hand, the contribution of red terms is $\leq K V_{red}$. On the other hand, the contribution of blue terms is $\leq V_{blue} \times v_{1}+Km$, where the $m$ is added to take 
 into account a possible last segment colored in blue only
 because it has reached the index $n$. Writing that $n=V_{red}+V_{blue}$, we deduce that 
 $v_{2} n \leq K V_{red} + v_{1}(n-V_{red}) + Km$, so that $V_{red} \geq \textstyle{ \frac{v_{2}-v_{1}}{K-v_{1}}}n -Km/(K-v_{1})$.
  Then use the fact that at least $V_{red}/m$ terms belong to $I$.
 \end{proof} 
 
In~\cite{Pem}, the result corresponding to our Lemma~\ref{l:beaucoup-de-bien}, and an estimate of the type given by Theorem~\ref{t:borne-bien},
are used in combination with an elaborate second moment argument.
In the present context, the following first moment argument turns out to be sufficient. 

\begin{proof}[Proof of the lower bound part of Theorem~\ref{t:BD}]
Assume that $X^{N}_{0}=N \delta_{0}$.
Let $\beta>0$ and $\theta>0$ be as in Theorem~\ref{t:borne-bien}. Then let $\lambda>0$, and define
$$m:= \left\lceil \theta^{3/2} \left(    \frac{(1+\lambda) (\log N)}{ (\chi(p) - \beta)^{1/2} }  \right)^{3}  \right\rceil,$$
and $\epsilon:=\theta/m^{2/3}$, so that, by  Theorem~\ref{t:borne-bien},
 \begin{equation}\label{e:appli-th-3}\rho(m,\epsilon) \leq N^{-(1+\lambda)} \mbox{ for all large } N.\end{equation}
Then let $0<\gamma<1$ and define $v_{2}:=v(p)-(1-\gamma) \epsilon$ and $v_{1}:=v(p) - \epsilon$.

 Let also $n:=\floor{N^{\xi}}$ for some $0<\xi<\lambda$.
Now consider $\kappa>0$, and let $K:=\kappa \log (2Nn)$. Consider the maximum of the random walk steps 
performed during the branching steps of $(X^{N}_{k})_{k \geq 0}$ between time $0$ and time $n$. There are $2Nn$ 
such steps, so that, by assumption (A2), there exists a value of $\kappa$ such that the probability that this maximum is larger than or equal to $K$ is less than 
$(2N n)^{-2008}$ for all large enough $N$.   
Now denote by $B_{n}$ the number of vertices in $G^{N}_{0} \cup \cdots \cup G^{N}_{n}$ (see Section~\ref{ss:cwafonbrw}) 
that are $(m,v(p)-\epsilon)-$good (each with respect to the $\BRW_{i}$ it belongs to).
 Observe that, with our definitions, for $N$ large enough, $ \textstyle{ \frac{v_{2}-v_{1}}{K-v_{1}}}\frac{n}{m} -K/(K-v_{1}) > 0$.
As a consequence, using Lemma~\ref{l:beaucoup-de-bien}, we see that, for $N$ large enough, 
the event $\max X^{N}_{n} \geq v_{2} n$ implies that either there exists a random walk step between 
time $0$ and $n$ which is $\geq K$, or $B_{n} \geq 1$. Using the union bound and the above estimate, we deduce that
 \begin{equation}\label{e:borne-reunion}\P \left( \max X^{N}_{n} \geq v_{2} n  \right) \leq (2N n)^{-2008} + \P(B_{n} \geq 1).\end{equation}
  On the other hand, $B_{n}$ can be written as 
\begin{equation}\label{e:decomp}B_{n}:=\sum_{u \in \TT_{1} \cup \cdots \cup \TT_{N}} \un(\mbox{  $u$ is $(m,v(p)-\epsilon)-$good}) \un( u \in G^{N}_{0} \cup \cdots \cup G^{N}_{n}).\end{equation}
Now observe that, by definition, for a vertex $u$ at depth $\ell$, the 
event $u\in G^{N}_{0} \cup \cdots \cup G^{N}_{n}$ is measurable with respect to the random walk increments performed 
at depth at most $\ell$, that is, the family of random variables
$\Phi_{i}(w)-\Phi_{i}(v)$, where $i \in \ig 1,N \id$, $v,w \in \TT_{i}$, $w$ is a child of $v$ (with respect to the tree structure of $\TT_{i}$), 
and $w,v$ are both located at a depth $\leq \ell$ in $\TT_{i}$.  On the other hand, the event that 
 $u$ is $(m,v(p)-\epsilon)-$good is measurable
 with respect to the random walk increments performed 
at depth at least $\ell$, that is, the family of random variables
$\Phi_{i}(w)-\Phi_{i}(v)$, where $i \in \ig 1,N \id$, $v,w \in \TT_{i}$, $w$ is a child of $v$, 
and $w,v$ are both located at a depth $\geq \ell$ in $\TT_{i}$.
 
As a consequence, the two events  $\{ u\in G^{N}_{0} \cup \cdots \cup G^{N}_{n} \}$ and  $\{ u\mbox{ is $(m,v(p)-\epsilon)-$good} \}$ are independent.
Since the total number of vertices in $G^{N}_{0} \cup \cdots \cup G^{N}_{n}$ is equal to $N (n+1)$, we 
deduce from~(\ref{e:appli-th-3}) and~(\ref{e:decomp}) that
$\E \left(     B_{n}    \right)  \leq  N(n+1) N^{-(1+\lambda)}$. Using Markov's inequality, we finally deduce from~(\ref{e:borne-reunion}) that
 \begin{equation}\label{e:encore-une-equation}\P( \max X^{N}_{n} \geq v_{2} n  ) \leq (2N n)^{-2008} + (n+1) N^{-\lambda}.\end{equation}

Now start with the obvious inequality, valid for all $t$, $\exp(t \max X^{N}_{n}) \leq \sum_{i=1}^{N}\sum_{u \in \TT_{i}(n)} \exp( t \Phi_{i}(u))$.
Taking expectations, we deduce that
$\E ( \exp(t \max X^{N}_{n}) ) \leq N 2^{n} \exp(n \Lambda(t))$. Using the definition of $t^{*}$ and $v(p)$, we then obtain that
\begin{equation}\label{e:ineg-exp}\E( \exp (t^{*} (\max X^{N}_{n}   - v(p) n)    )  \leq  N.\end{equation} 
Using~(\ref{e:ineg-exp}), we deduce that\footnote{Here are the details. Let $M:=\max X^{N}_{n} - v(p)n$.  From the fact that, for all large enough $x$, 
$x \leq \exp(t^{*}x/2008)$, we deduce that $\E(M \un( M \geq bn ) )   \leq \E \exp(t^{*}M - \frac{2007}{2008}t^{*}bn) $ for all large enough $n$. Similarly, 
    $\E(v(p)n \un( M \geq bn ) )   \leq |v(p)|n\E \exp(t^{*}M - t^{*}bn) $. The result follows from summing the two inequalities above and applying~(\ref{e:ineg-exp}).}, 
    for all $b>0$, and all large enough $n$, 
\begin{equation}\label{e:encore-une-equation-2}\E \left[  \max X^{N}_{n}  \un(   \max X^{N}_{n}    \geq (v(p)+b)n     )      \right]  \leq N 
\exp\left(  - \textstyle{\frac{2007}{2008}} t^{*} b n  \right)(1+|v(p)|n).\end{equation}
  Now observe that, by definition, $\E(n^{-1}\max X^{N}_{n}) $ is bounded above by
  $$v_{2} + (v(p)+b)\P( \max  X^{N}_{n} \geq v_{2} n) + n^{-1}\E\left[   \max X^{N}_{n} \un( \max X^{N}_{n} \geq (v(p)+b) n ) \right].$$
 Choosing a $b>0$ , we deduce from~(\ref{e:encore-une-equation}), (\ref{e:encore-une-equation-2}), and the definition of $v_{2}$, that, for all large enough $N$, 
 $$\E(n^{-1}\max X^{N}_{n}) \leq (v(p)-(1-\gamma)\epsilon) + o((\log N)^{-2}).$$ Using subadditivity 
 (see the proof of Proposition~\ref{p:subadd}), we have that $v_{N}(p) \leq \E(n^{-1}\max X^{N}_{n})$, and we easily deduce that
  $$v_{N}(p) \leq (v(p) -(1-\gamma)\epsilon) + o((\log N)^{-2}).$$
 Now remember that, as $N$ goes to infinity, $\epsilon \sim \frac{\chi(p) - \beta}{(1+\lambda)^{2}}  (\log N)^{-2}$.
 Since the above estimates are true for arbitrarily small $\beta$, $\lambda$ and $\gamma$, the conclusion follows.
 \end{proof}

\section{The upper bound}\label{s:upper}

The proof of the upper bound on $v_{\infty}(p)-v_{N}(p)$ given in~\cite{Ber} was in 
some sense a rigorous version of the heuristic argument of Brunet and Derrida 
according to which we should compare the behavior of the particle system with $N$ particles, 
to a version of the infinite population limit dynamics suitably modified by a cutoff. 
The proof given here relies upon a direct comparison with branching random walks, using the fact that, above the threshold induced by the selection 
steps, the behavior of our branching-selection particle system is exactly that of a branching random walk.

Consider $0<\lambda<1$ and let $\epsilon := \frac{\chi(p)}{((1-\lambda )\log N)^{2}}$. With this choice of $\epsilon$, 
as $N$ goes to infinity, Theorem~\ref{t:borne-bien-original} yields that 
\begin{equation}\label{e:grand-rho}\rho(\infty,\epsilon) = N^{-(1-\lambda)+o(1)}.\end{equation}

Let us now quote the following result, which is a consequence of Theorem 2 Section 6 Chapter 1 in \cite{AthNey}.
\begin{lemma}\label{l:athreya-ney}
Let $(M_n)_{n \geq 0}$ denote the population size of a supercritical Galton-Watson process with square-integrable offspring distribution started with $M_0=1$. 
 Then there exist $r>0$ and $\phi>1$ such that, for all $n \geq 0$, 
$$ P(M_n \geq \phi^n) \geq r.$$
\end{lemma}

Let $R$ be such that $R < v(p)$ and $p([R, +\infty)) \geq 2/3$. Consider a Galton-Watson tree whose offspring distribution is a binomial with parameters $2$ and $p([R, +\infty))$. The average number of offspring is thus equal to $2 p([R, +\infty)) \geq 4/3 >1$ with our assumptions. 
In the sequel, we use the notations $r$ and $\phi$ to denote the numbers given by Lemma~\ref{l:athreya-ney} when we use this offspring distribution.

Now, let $s_N := \ceil{\frac{\log N}{\log \phi}}+1$, consider $0<\eta<1$, and define $m:= \ceil{\textstyle{\frac{(v(p)-R) s_{N}}{  \eta \epsilon }}}$ and $n:=m+s_{N}$. Let $u$ denote a vertex at depth $m$ in a branching random walk $\BRW$, and assume that $ \Phi(u) \geq (v(p)-\epsilon) m$. 
Consider the probability that, conditional upon the values of $\Phi$ on the vertices located at depth at most $m$, there are at least $\phi^{s_N}$ distinct descending paths $u=:u_{m},\ldots, u_{n}$ starting at $u$ and satisfying $u_{i+1}-u_i \geq R$ for all 
$i \in \ig m, n-1 \id$. Lemma~\ref{l:athreya-ney} above shows that this probability is $\geq r$. Moreover, with our definition  of $m$ and $n$, and our assumption on the value of $\Phi(u)$, any such descending path has the property that  $\Phi(u_{i}) \geq  (v(p)-\epsilon(1+\eta)) i$
for all $i \in \ig m, n \id$. We conclude that the probability that there exist at least $\phi^{s_N}$ distinct descending paths of the form
 $root=u_{0},\ldots, u_{n}$ such that $\Phi(u_{i}) \geq  (v(p)-\epsilon(1+\eta)) i$
for all $i \in \ig 0, n \id$, is $\geq  \rho(m,\epsilon)  r$.

Now define $A$ as the event that, for all $j \in \ig 1, N \id$,  $\mbox{BRW}_{j}$ does not contain more than $\phi^{s_N}$ distinct descending paths of the form
 $root=u_{0},\ldots, u_{n}$ such that $\Phi_j(u_{i}) \geq  (v(p)-\epsilon(1+\eta)) i$
for all $i \in \ig 0, n \id$.  Using the fact that  $\mbox{BRW}_{1},\ldots, \mbox{BRW}_{N}$ are independent and the above discussion, we see that 
$$\P(A) \leq \left[1-  \rho(m,\epsilon) r  \right]^N.$$
Using~(\ref{e:grand-rho}), the obvious inequality $\rho(m, \epsilon) \geq \rho(\infty,\epsilon)$, and the fact that $1-x \leq \exp(-x)$ for all $x$, 
we deduce that
\begin{equation}\label{e:grand-rho-2}\P(A) \leq \exp(-N^{\lambda+o(1)}).\end{equation}

Let $\delta:= \epsilon (1+\eta)$. Define the event
 $B:=\{  \min(X^{N}_{k}) <     (v(p)-\delta) k     \mbox{ for all } k \in \ig 1 , n \id \}$, and  assume that $B \cap A^{c}$ occurs.
 From the definition of the selection mechanism, we conclude that there must be at least $\phi^{s_{N}}$ distinct vertices in the set $G^{N}_{n}$, 
 which is a contradiction since $\phi^{s_{N}}>N$.
  As a consequence, $B \cap A^{c} = \emptyset$, so that $B \subset A$.

From~(\ref{e:grand-rho-2}), we thus obtain that 
   \begin{equation}\label{e:borne-prob-B}\P(B)  \leq \exp(-N^{\lambda+o(1)}).\end{equation}
 
 To exploit this bound, we use the following result.
 \begin{prop}\label{p:borne-inf-B}
With the previous notations, for all $N$ large enough, 
$$v_{N}(p) \geq (v(p)-\delta) -  |v(p)-\delta| n\P(B)-n\E(|  \Theta_n | \un(B)  ),$$
where $\Theta_n$ is the minimum of $2nN$ i.i.d. random variables with distribution $p$.
\end{prop}

\begin{proof}
We re-use the coupling construction given in the proof of Proposition~\ref{p:subadd}, and assume that $(X_n^N)_{n \geq 0}$ is defined using this construction by the identity $X^N_n := W^N_{0,n}$.
Start with $\Gamma_{0}:=0$ and $J_{0}:=0$, and $i:=0$.
Given $i \geq 0$, $\Gamma_{i}$ and $J_{i}$, let $L_{i+1}:=\inf \{   k \in \ig 1 , n  \id; \  \min(W^{N}_{\Gamma_{i},k}) \geq     (v(p)-\delta) k   \}$, 
with the convention that $\inf \emptyset := n$.
Then let $\Gamma_{i+1}:=\Gamma_{i}+L_{i+1}$, and let $J_{i+1}:=J_{i} + \min(W^{N}_{\Gamma_{i},L_{i+1}})$.

Using an argument similar to the proof of Lemma~\ref{l:kernel-monot}, it is then quite easy to deduce that, a.s., 
\begin{equation}\label{e:minoration} \mbox{ for all $i \geq 0$,  } \min W^{N}_{0,\Gamma_{i}} \geq J_{i}.\end{equation}

Observe that the sequence $(\Gamma_{i+1}-\Gamma_{i})_{i \geq 0}$ is i.i.d., and that the common distribution of 
the $\Gamma_{i+1}-\Gamma_{i}$ is that of the random variable $L$ defined by 
$L:= \inf \{   k \in \ig 1 , n  \id; \  \min(X^{N}_{k}) \geq     (v(p)-\delta) k   \}$, with the convention that $\inf \emptyset := n$.
Similarly, the sequence $(J_{i+1}-J_{i})_{i \geq 0}$ is i.i.d., the common distribution of the $J_{i+1}-J_{i}$ being that of 
$\min X^{N}_{L}$.

From the law of large numbers and  Proposition~\ref{p:subadd}, we have that, a.s.,  
$\lim_{i \to +\infty} i^{-1} \min X^{N}_{\Gamma_{i}} = v_{N}(p) \E(L)$, while the law of large numbers and~(\ref{e:minoration}) imply that
  $\liminf_{i \to +\infty} i^{-1} \min X^{N}_{\Gamma_{i}} \geq  \E(\min X^{N}_{L})$.
We conclude that $v_{N}(p) \geq \textstyle{\frac{\E(\min X^{N}_{L})}{\E(L)}}$.
Now, let $\Theta_n$ denote the minimum of all the random walk steps performed by the branching-selection system between time $0$ and $n$.

By definition we have that $\min X^{N}_{L} \geq (v(p)- \delta) L \un(B^{c}) + L \Theta_n   \un(B)$, so that 
$\E(\min X^{N}_{L}) \geq (v(p)-\delta) (\E(L) - \E(L \un(B)))  +  \E(   L \Theta_n  \un(B)    )$.
Using the fact that $1 \leq L \leq n$, we obtain that $\textstyle{\frac{\E(\min X^{N}_{L})}{\E(L)}} \geq  (v(p)-\delta) -  |v(p)-\delta| n\P(B)-n\E( | \Theta_n | \un(B)  )$.
\end{proof}

\begin{proof}[Proof of the upper bound part in Theorem~ \ref{t:BD}]
In view of Proposition~\ref{p:borne-inf-B}, we deduce that 
$v_{N}(p) \geq (v(p)-(1+\eta)\epsilon) (1-n\P(B)) - n\E(  |\Theta_n | \un(B)  ).$
Bounding above $ | \Theta_n |$ by the sum of the absolute values of the $2nN$ corresponding i.i.d. variables, and using Schwarz's inequality thanks to Assumptions (A1) and (A2), 
we deduce that $\E(  |\Theta_n |\un(B)  ) \leq 2 n N C \P(B)^{1/2}$ for some constant $C$ (depending only on $p$).
From~(\ref{e:borne-prob-B}) and the definition of $n$, we deduce that, as $N$ goes to infinity,  
$n\P(B)$ and $n \E( | \Theta_n | \un(B)  )$  are $o((\log N)^{-2})$, so we obtain that 
$$v_{N}(p) \geq v(p)   -   \frac{\chi(p)(1+\eta)}{(1-\lambda)^{2}} (\log N)^{-2} + o( (\log N)^{-2}).$$
Since $\lambda$ and $\eta$ can be taken arbitrarily small in the argument leading to the above identity, the conclusion follows.

\end{proof}

\section{The Bernoulli case when $1/2 \leq \alpha < 1$}\label{s:autres-cas}

In the Bernoulli case $p=\alpha \delta_{1} + (1-\alpha) \delta_{0}$, 
with $1/2 \leq \alpha<1$, Assumption (A3) breaks down, and 
the behavior of the particle system  turns out to be quite different from Brunet-Derrida, as stated in the following theorems. 
 Note that, when  $1/2 \leq \alpha < 1$, $v_{\infty}(p)=1$.

\begin{theorem}\label{t:un-demi}
For $\alpha=1/2$, there exists $0<c_{*}(p) \leq c^{*}(p)<+\infty$ such that, for all large $N$, 
\begin{equation}\label{e:un-demi}  c_{*}(p) N^{-1} \leq 1 - v_{N}(p) \leq c^{*}(p) N^{-1}.\end{equation}
\end{theorem}

\begin{theorem}\label{t:sup-un-demi}
For $\alpha>1/2$, there exists $0<d^{*}(p) \leq d_{*}(p)<+\infty$ such that, for all large $N$, 
\begin{equation}\label{e:sup-un-demi}  \exp(-d_{*}(p) N) \leq   1 - v_{N}(p) \leq \exp(-d^{*}(p) N).\end{equation}
\end{theorem}

\subsection{Lower bound when $\alpha=1/2$}

It is easily checked that, for all $m \geq 0$, the number 
of particles in the branching-selection system that are located at position $m$ after $m$ steps, that is, $X^{N}_{m}(m)$, 
is stochastically dominated by the total population at the $m-$th
generation of a family of $N$ independent Galton-Watson trees, with offspring distribution binomial$(2,1/2)$.  This corresponds to 
the critical case of Galton-Watson trees, and the probability that such a tree survives up to the $m-$th 
generation is $\leq c m^{-1}$ for some constant $c>0$
and all large $m$. As a consequence, the union bound over the $N$ Galton-Watson trees yields that, for large enough $m$, 
$\P(X^{N}_{m}(m) \geq 1) \leq cNm^{-1}$.
On the other hand, we have by definition that $\E \max(X^{N}_{m}) \leq  m  \P(X^{N}_{m}(m) \geq 1) +  (m-1)  \P(X^{N}_{m}(m) = 0)$.
Choosing $m := A N$, where $A \geq 1$ is an integer, we deduce that, for large $N$,
$    m^{-1}\E \max(X^{N}_{m}) \leq 1 - \textstyle{\frac{1}{AN}}(1-c/A)$. 
Using subadditivity (see the proof of Proposition~\ref{p:subadd}), we have that $v_{N}(p) \leq \E(m^{-1}\max X^{N}_{m})$. 
The lower bound in~(\ref{e:un-demi}) follows by choosing $A>c$. 
\subsection{Upper bound when $\alpha=1/2$}

Given $m\geq 1$, define $U:=\inf \{ n \in \ig 1, m \id ; \ X^{N}_{n}(n) \leq 2N/3 \}$, with the convention that
 $\inf \emptyset := m$. Observe that $\min X^{N}_{U} \geq U-1$, since, by definition, $X^{N}_{U-1}(U-1) \geq 2N/3$, so that, 
 after the branching step applied to  $X^{N}_{U-1}$, the number of particles whose positions are $\geq U-1$ must be  $ \geq 2 \times 2N/3$, whence
 $\geq N$. 
 
Using an argument similar to the proof of Proposition~\ref{p:borne-inf-B}, we deduce that
\begin{equation}\label{e:min-un-demi} v_{N}(p) \geq   1 - \frac{1}{\E(U)}.\end{equation}
 
 The lower bound  in~(\ref{e:un-demi}) is then a direct consequence of the following claim.
{\bf Claim:} for small enough $\epsilon>0$, with $m:=\floor{\epsilon N}$,  there exists $c(\epsilon)>0$ such that
$\E(U) \geq  c(\epsilon)N$ for all large $N$. To prove the claim, introduce for every $x \in \N$ the Markov chain 
$(V^{x}_{k})_{k \geq 0}$ defined by the initial condition $V^{x}_{0}:=x$, and the following transitions: 
given $V^{x}_{0},\ldots, V^{x}_{k}$, the next term $V^{x}_{k+1}$ is the minimum of $N$ and of a random variable with a binomial$(2V^{x}_{k},1/2)$
 distribution. Clearly, the sequences $(V^{N}_{k})_{k \geq 0}$ and  $(X^{N}_{k}(k))_{k \geq 0}$ have the same distribution.
 Moreover,  given two starting points $x,y \in \N$ such that $x \leq y$, one can easily couple $(V^{x}_{k})_{k \geq 0}$ and $(V^{y}_{k})_{k \geq 0}$
in such a way that $V^{x}_{k} \leq V^{y}_{k}$ for all $k \geq 0$. As a consequence,  choosing $x_{N}:=\floor{3N/4}$, we see that $U$ stochastically 
dominates the random variable $T$ defined by $T:=\inf \left\{ n \in \ig 1, m \id ; \ V^{x_{N}}_{n} \leq 2N/3 \right\}$ (again with $\inf \emptyset := m$),
 so that $\P(U=m) \geq P(T=m)$.

 Now let us define yet another Markov chain $(Z_{k})_{k \geq 0}$ by $Z_{0}:=x_{N}$ and the following transitions: 
 given $Z_{0},\ldots, Z_{k}$, the next term $Z_{k+1}$ is a random variable with a binomial$(2Z_{k},1/2)$
 distribution. Clearly we can couple $(Z_{k})_{k}$ and $(V^{x_{N}}_{k})_{k}$ so that they coincide up to one unit of time before 
 the first hitting of $\ig N, +\infty \ig$. As a consequence, the two events  
  $A_{1}:=\{    \sup_{k \in \ig 0 , m \id}     |   V^{x_{N}}_{k}  -  \floor{3N/4}     |   < N/16    \}$
and $A_{2}:=\{    \sup_{k \in \ig 0 , m \id}     |   Z_{k}  -  \floor{3N/4}     |   < N/16    \}$ have the same probability. 
Observing that $(Z_{k})_{k \geq 0}$ is a martingale, we can use Doob's maximal inequality to prove that
  $P \left(  A_{1}^{c}  \right) = P \left(  A_{2}^{c}  \right) \leq E (Z_{m} - \floor{3N/4})^{2}   (N/16)^{-2}$.
  Then, it is easily checked from the definition that $E(Z_{k+1}^{2} | Z_{k}) = Z_{k}^{2} + Z_{k}/2$ for all $k \geq 0$, and, 
  using again the fact that $(Z_{k})_{k \geq 0}$ is a martingale,
   we deduce that  $E (Z_{m} - \floor{3N/4})^{2} \leq m N/2$. 
   As a consequence, we see that, choosing $\epsilon>0$ small enough, 
 we can ensure that   $P \left(  A_{1}^{c}  \right) \leq 1/2008$
 for all large $N$. Since, by definition, $A_{1}$ implies that $T=m$, 
 we finally deduce that, for such an $\epsilon$, and all $N$ large enough, we have that 
 $\P(U = m) \geq P(T=m) \geq 2007/2008$. The conclusion follows.
 
 \subsection{Upper and lower bound when $1/2<\alpha<1$}

As for the lower bound, observe that the probability that all the $2N$ particles generated during a branching step remain at the position from which they
originated is $(1-\alpha)^{2N}$, so that $\E( \max X^{N}_{n} ) \leq n (1 - (1-\alpha)^{2N})$. 
As for the upper bound, observe that, starting from $N$ particles at a site, the number of particles generated from these during a branching step and 
that perform  $+1$ random walk steps has a binomial$(2N,\alpha)$ distribution, whose expectation is $2 \alpha N$, with $2\alpha > 1$. Using a standard large deviations bound
for binomial random variables, we see that the probability for this number to be less than $N$ is $\leq \exp(-c N)$ for some $c>0$. 
 Using superadditivity $\E( \min X^{N}_{n} )$ (derived in exactly the same way as the subadditivity property of $\E( \max X^{N}_{n} )$, see the proof of Proposition~\ref{p:subadd}), 
 it is easy to deduce that $\E( \min X^{N}_{n} ) \geq n (1 - \exp(-c N))$. 
 The result follows.

\section{Discussion}\label{s:discussion}

This section contains a discussion whose goal is to  provide a self-contained qualitative explanation of the $(\log N)^{-2}$ order of magnitude 
appearing in Theorem~\ref{t:BD}. Most of the discussion consists in 
explaining the $\epsilon^{-1/2}$ scaling of $\log \rho(\infty, \epsilon)$, and of $\log \rho(m,\epsilon)$
when $m \propto \epsilon^{-3/2}$, 
in a way that is (hopefully) less technically demanding than the proofs presented 
in~\cite{GanHuShi}, although we follow the proof strategy 
of~\cite{GanHuShi} rather closely. Note that the discussion here deals mostly with the order of magnitude of terms, not with the precise value 
of the constants as in~\cite{GanHuShi}. For the sake of readability, some of the arguments are only 
discussed in a quite informal way.

\subsection{Asymptotic behavior of $\rho(\infty,\epsilon)$ and $\rho(m,\epsilon)$}

\subsubsection{Connection between $\rho(\infty,\epsilon)$ and $\rho(m,\epsilon)$}

A first remark is that the asymptotic behavior of $\rho(\infty,\epsilon)$ can be connected with that of quantities 
of the form $\rho(m,\epsilon)$ under appropriate conditions.
One obvious inequality, valid for all $m \geq 0$, is the following 
\begin{equation}\label{e:compare-evident}\rho(m,\epsilon) \geq \rho(\infty,\epsilon).\end{equation}

In the reverse direction, we have the following.
 \begin{prop}\label{p:comparaison}
 There exist $R < v(p)-1$, $\phi>1$, $r>0$ and $c>0$, depending only on $p$, such that, for all $m \geq 0$, and all $0<\epsilon<1$, the condition
\begin{equation}\label{e:assez-grand}\phi^{q} \rho(m, (1-\alpha) \epsilon) \geq c\end{equation} implies that the following inequality holds 
\begin{equation}\label{e:compare-fini-infini}\rho(\infty,\epsilon)  \geq \frac{r}{2} \rho(m,(1-\alpha) \epsilon),\end{equation}
where $\alpha$ and $\epsilon$ are arbitrary numbers satisfying $0<\alpha<1$ and $\epsilon>0$,   
and
$$q:= \left\lfloor \left( \frac{\alpha \epsilon m}{v(p)-\epsilon-R} \right) \right\rfloor.$$  
\end{prop}   

The proof of the above proposition uses the following elementary lemma.
\begin{lemma}\label{l:GW-extinct}
Consider a Galton-Watson process with offspring distribution $Q$. 
If there exists $a \geq 1$ such that $a \times Q([a,+\infty[) \geq 2 \log 2$, then the survival probability is larger than or equal to $Q([a,+\infty[)/2$. 
\end{lemma}
\begin{proof}[Proof of Lemma~\ref{l:GW-extinct}]
Let $g(s):=\sum_{k=0}^{+\infty} Q(k) s^{k}$ for $s \in [0,1[$. 
By coupling, it is enough to prove the result under the additional assumption that only the values $0$ and $a$ have non-zero probability with respect to $Q$, so we may assume  
that $g(s)=1-Q(a)+s^{a} Q(a)$. 
Since $a Q(a) > 1$, we have a super-critical Galton-Watson process, and, from standard  theory, we know that the extinction probability $d$ of the process 
is the unique solution in $[0,1[$ of the equation $g(d)=d$, with $g(s) > s$ for $s \in ]0,d[$ and $g(s)<s$ for $s \in ]d,1[$.
Our assumption that $a \times Q(a) \geq 2\log 2$ easily yields the fact that $g(1-Q(a)/2) \leq 1-Q(a)/2$, whence the fact that $d$ must be $\leq 1-Q(a)/2$.
 The result follows.
\end{proof}

\begin{proof}[Proof of Proposition~\ref{p:comparaison}]

Consider the values of $R, \phi, r$ defined in Section \ref{s:upper}, in the argument following Lemma  \ref{l:athreya-ney}.
 Then consider 
a descending path $root=u_{0},\ldots, u_{m+q} \in \TT$ such that $\Phi(u_{i}) \geq (v(p)-(1-\alpha)\epsilon) i$ for all $i \in \ig 0,m\id$, and $\Phi(u_{i+1})-\Phi(u_i) \geq R$ for all 
$i \in \ig m, m+q-1 \id$. 
We see from the definition of $q$ that $\Phi(u_{i}) \geq (v(p)-\epsilon) i$  for all $i \in \ig 0,m+q\id$. 

We now define a Galton-Watson branching process of vertices of $\TT$ in 
which, for all $n$, the $n-$th generation of the process is formed by vertices in $\TT((m+q)n)$. 
First, the zero-th generation of the process is formed by the root of $\TT$. Then, given 
a vertex $x \in \TT((m+q)n)$ belonging to the $n-$th generation of the process, the offspring of this vertex in the branching process is formed by all the endpoints $y$ 
of descending paths $x=:u_{0},\ldots, u_{m+q}:=y$ in $\TT$ 
such that  $\Phi(u_{i})-\Phi(u_{0}) \geq (v(p) - \epsilon) i$ 
for all $i \in \ig 0,m+q \id$. From the definition of the branching mechanism of $\BRW$, we see that we have defined
a Galton-Watson branching process.  Now, re-doing the argument following Lemma~\ref{l:athreya-ney} in Section~\ref{s:upper},  
we see that the offspring distribution of this branching process gives at least $\phi^{q}$ 
children with probability at least $\rho(m,(1-\alpha)\epsilon) r$. 
On the other hand, the definition of our branching process shows that if it never goes extinct, the root of $\TT$ is $(\infty,v(p)-\epsilon)-$good.
As a consequence, the survival probability of our process is a lower bound for $\rho(\infty,\epsilon)$. The result then follows from Lemma~\ref{l:GW-extinct}, choosing
$c>(2 \log 2)/r$.
\end{proof}

We conclude this section by the following remark: using the above results, it is possible to deduce the conclusion of Theorem~\ref{t:borne-bien} from the conclusion of 
Theorem~\ref{t:borne-bien-original}. In~\cite{GanHuShi}, Theorem~\ref{t:borne-bien} is in fact an intermediate step in the proof of 
Theorem~\ref{t:borne-bien-original}, so our remark does not lead to an alternative way of proving~Theorem \ref{t:borne-bien} from first principles. However, its interest is to show that, 
as soon as $\epsilon^{1/2} \log \rho(\infty,\epsilon)$ converges to some limit, this limit can be approached arbitrarily closely by  expressions of the form 
$\epsilon^{1/2} \log \rho(m,\epsilon)$, with $\epsilon = \theta/m^{2/3}$ for some large enough constant $\theta$.

\begin{proof}[Proof of Theorem~\ref{t:borne-bien} from the conclusion of Theorem~\ref{t:borne-bien-original}] 
Let $0<\alpha<1$, $\theta>0$ and $m\ge 1$. 
Set $\epsilon:=\theta/m^{2/3}$.
Let $R,\phi, r, c, q$ be defined as in the statement of Proposition \ref{p:comparaison}.
As $m$ goes to infinity, we have that
$
\log \phi^q \sim \theta\alpha \log(\phi) (v(p)-R)^{-1} m^{1/3}
$, 
and, by Theorem~\ref{t:borne-bien-original},
$$
\log \rho(\infty,(1-\alpha)\epsilon) \sim -\chi(p)^{1/2}\theta^{-1/2}(1-\alpha)^{-1/2}m^{1/3}.
$$
Therefore, provided that $\theta$ has been chosen large enough,  Inequality~(\ref{e:assez-grand})  holds
for large enough $m$. Given such $\theta$ and $m$, Proposition~\ref{p:comparaison} and Theorem~\ref{t:borne-bien-original} yield that
$$
\rho(m,(1-\alpha)\epsilon) \le (2/r) \rho(\infty,\epsilon) = (2/r) \exp\left(-\left[\frac{\chi(p)+o(1)}{\epsilon} \right]^{1/2}\right).
$$
Setting $\widetilde{\epsilon}:=(1-\alpha)\theta m^{-2/3}$, one obtains that
$$
\rho(m,\widetilde{\epsilon}) \le \exp\left(-\left[\frac{\chi(p)(1-\alpha)+o(1)}{\widetilde{\epsilon}} \right]^{1/2}\right).
$$
Since $\alpha$ can be chosen arbitrarily small, the conclusion follows.
\end{proof}
 
\subsubsection{Strategy and results}

Given the results of the previous section, the strategy consists in studying the order of magnitude of 
 $\log \rho(m,\epsilon)$, when $m$ has the scaling form $m \propto \epsilon^{-u}$ for some $u>0$, seeking a value of $u$
 such that~(\ref{e:assez-grand}) is satisfied. We shall see that $\log \rho(m,\epsilon) \propto \epsilon^{-h(u)}$, with $h(u)=u/3$ for all $0 <u \leq 3/2$, while the integer $q$
 in Proposition~\ref{p:comparaison} satisfies $q \propto \epsilon^{-(u-1)}$ for all $u>1$. For $u^{*}:=3/2$, the identity $h(u^{*})=u^{*}-1$ holds, 
 so that $q$ and $\log \rho(m,(1-\alpha)\epsilon)$ have the same order of magnitude and~(\ref{e:assez-grand}) can be satisfied. 
 Then~(\ref{e:compare-evident}) and~(\ref{e:compare-fini-infini}) imply that
$\log \rho(\infty,\epsilon) \propto -\epsilon^{-h(u^{*})} = -\epsilon^{-1/2}$, and (using the fact that $h$ is non-decreasing), $h(u)=h(u^{*})$ for all $u \geq u^{*}$.
Thus, the $\epsilon^{-1/2}$ scaling exponent is "explained" by $3/2$ being the solution of the equation $h(u^{*})=u^{*}-1$.

We deduce from these results  the existence of two distinct regimes for $\rho(m,\epsilon)$ with $m \propto \epsilon^{-u}$:
\begin{itemize}
\item when $0<u \leq 3/2$, $\log(\rho(m,\epsilon)) \propto -m^{1/3}$;
\item when $u \geq 3/2$, $\log(\rho(m,\epsilon)) \propto \log(\rho(\infty,\epsilon)) \propto -\epsilon^{-1/2}$.
\end{itemize}

\subsubsection{Asymptotics of $\log \rho(m,\epsilon)$ with $m \propto \epsilon^{-u}$, $0<u<3/2$}

Let us now explain how to compute $h(u)$, and assume throughout this section that $m \propto \epsilon^{-u}$, whence $\epsilon m \propto \epsilon^{1-u}$. 
A key idea is to perform a change of measure, replacing the step distribution $p$ of the $\BRW$ by the distribution $\tilde{p}$ defined by (see Section~\ref{ss:main})
$$\frac{d \tilde{p}}{d p}(x) := \frac{\exp{(t^{*}x)}}{\exp(\Lambda(t^{*}))}.$$ The mean value of a step with respect to $\tilde{p}$
is now equal to $v(p)$, and, if  $(S_{k})_{k \geq 0}$ denotes a random walk started at $S_{0}:=0$, with  i.i.d. increments whose common distribution is $p$ with respect to a probability measure $P$, 
and $\tilde{p}$ with respect to a probability measure $\tilde{P}$, the following identity holds for all $k$:
\begin{equation}\label{e:change-de-mes}2^{k}P\left[ (S_{0},\ldots, S_{k}) \in \cdot  \right]  = \tilde{E}\left[ e^{-t^{*}(S_{k} - v(p)k)}  
\un(S_{0},\ldots, S_{k}) \in \cdot  \right],\end{equation} 
where $\tilde{E}$ denotes expectation with respect to $\tilde{P}$. 
The raison d'\^etre of Assumption (A3) is to allow for such a change of measure.

Remember that $\rho(m,\epsilon)$ is the probability that at least one descending path $root=:u_{0},u_{1},\ldots, u_{m}$ exists in $\TT$ such that 
\begin{equation}\label{e:condition-1}\mbox{ for all $i \in \ig  0,m\id$, } \Phi(u_{i}) \geq (v(p)-\epsilon) i.\end{equation}
Observe that, for such a path, either 
\begin{equation}\label{e:condition-2}\mbox{ for all $i \in \ig  0,m\id$, } \Phi(u_{i}) \leq v(p)i + \epsilon^{-u/3},\end{equation}
or
\begin{equation}\label{e:condition-3}\mbox{ there exists  $i \in \ig  0,m\id$ such that } \Phi(u_{i}) > v(p)i + \epsilon^{-u/3}.\end{equation}
Denoting by $\Xi_{m}$ the (random) number of descending paths satisfying~(\ref{e:condition-1}) and~(\ref{e:condition-2}), and by  
$\Delta_{m}$ the number of those satisfying~(\ref{e:condition-1}) and~(\ref{e:condition-3}), we see that 
$$  \rho(m,\epsilon) = \Q(  \{  \Xi_{m} \geq 1    \} \cup  \{  \Delta_{m} \geq 1 \}),$$
 so that, obviously,    
\begin{equation}\label{e:compare-prob-esp}\Q(\Xi_{m} \geq 1)  \leq \rho(m,\epsilon) \leq  \E(\Xi_{m}) +  \Q(\Delta_{m} \geq 1).\end{equation}

By definition, $$\E(\Xi_{m}) = 2^{m} P(  \mbox{ for all $i \in \ig  0,m\id$, $v(p) i - \epsilon i \leq S_{i} \leq v(p) i + \epsilon^{-u/3}$}),$$
which rewrites, using~(\ref{e:change-de-mes}), as  
\begin{equation}\label{e:change-de-mes-en-action}\tilde{E}\left[ e^{-t^{*}(S_{m} - v(p)m)}  \un(\mbox{ for all $i \in \ig  0,m\id$, $v(p) i - \epsilon i \leq S_{i} \leq v(p) i + \epsilon^{-u/3}$}) \right].\end{equation}

Since the only paths that contribute to the above expectation have $$v(p) m - \epsilon m \leq S_{m} \leq v(p)m + \epsilon^{-u/3},$$ we see that
\begin{equation}\label{e:borne-partie-exp}\E(\Xi_{m}) \leq e^{t^{*} \epsilon m } \tilde{P}\left[\mbox{ for all $i \in \ig  0,m\id$, $v(p) i - \epsilon i \leq S_{i} \leq v(p) i + \epsilon^{-u/3}$ }\right].\end{equation} 
and
\begin{equation}\label{e:borne-partie-exp-2}\E(\Xi_{m}) \geq e^{-t^{*} \epsilon^{-u/3} } \tilde{P}\left[\mbox{ for all $i \in \ig  0,m\id$, $v(p) i - \epsilon i \leq S_{i} \leq v(p) i + \epsilon^{-u/3}$ }\right].\end{equation} 

 Now observe that, under $\tilde{P}$, $(S_{k} - v(p) k)_{k \geq 0}$ is a random walk with centered square-integrable increments. Moreover,  
 $(\epsilon^{-u/3})^{2}<<m$ since $u>0$, and, $\epsilon^{-u/3}>>\epsilon m$ as soon as $u<3/2$. As a consequence, 
 the usual Brownian scaling for random walks yields that when $0<u<3/2$,  
\begin{equation}\label{e:borne-partie-rw} \log \tilde{P}\left[\mbox{ $ \forall i \in \ig  0,m\id$, $v(p) i - \epsilon i \leq S_{i} \leq v(p) i + \epsilon^{-u/3}$ }\right] 
\propto -\frac{\epsilon^{-u}}{(\epsilon^{-u/3})^{2}} = -\epsilon^{-u/3}.\end{equation}
Since $\epsilon^{-u/3}>>\epsilon m$ when $u<3/2$, we deduce from~(\ref{e:borne-partie-exp}) and~(\ref{e:borne-partie-exp-2}) that
  \begin{equation}\label{e:borne-esp-xi}\log \E(\Xi_{m}) \propto -\epsilon^{-u/3} \mbox{ for $1 < u < 3/2$}.\end{equation}

As for $\Delta_{m}$, the union bound yields that 
$$\Q(\Delta_{m} \geq 1) \leq \sum_{i=0}^{m} \Q(  \exists x \in \TT(i); \ \Phi(x) > v(p)i + \epsilon^{-u/3}),$$
whence
$$\Q(\Delta_{m} \geq 1) \leq \sum_{i=0}^{m}  2^{i} P( S_{i} > v(p)i + \epsilon^{-u/3}).$$
For all $i \in \ig 0,m \id$, the change of measure shows that
$2^{i} P( S_{i} > v(p)i + \epsilon^{-u/3}) = \tilde{E}\left[ e^{-t^{*}(S_{i} - v(p)i)}  \un( S_{i} > v(p) i + \epsilon^{-u/3}) \right] \leq \exp(-t^{*} \epsilon^{-u/3})$,
so we easily deduce that
\begin{equation}\label{e:borne-proba-depassement}-\log \Q(\Delta_{m} \geq 1) 
\mbox{ is at least } \propto -\epsilon^{-u/3}.\end{equation}

We conclude from~(\ref{e:borne-esp-xi}), ~(\ref{e:borne-proba-depassement}) and~(\ref{e:compare-prob-esp}) that 
\begin{equation}\label{e:une-moitie}-\log \rho(m,\epsilon) \mbox{ is at least }\propto \epsilon^{-u/3} \mbox{ for }0<u<3/2.\end{equation}

Using the same kind of argument that led to~(\ref{e:borne-esp-xi}), but working a little more (we omit the details), it is possible to show  that 
 \begin{equation}\label{e:borne-esp-xi-carre}\log \E(\Xi^{2}_{m}) \mbox{ is at most }\propto  \epsilon^{-u/3} \mbox{ for $0 < u < 3/2$}.\end{equation}

Then, we can use the classical second moment inequality:
$$\Q( \Xi_{m}>0) \geq \frac{\E(\Xi_{m})^{2}}{\E(\Xi_{m}^{2})},$$
to deduce from~(\ref{e:compare-prob-esp}), ~(\ref{e:borne-esp-xi}) and~(\ref{e:borne-esp-xi-carre}) that 
\begin{equation}\label{e:une-autre-moitie}-\log \rho(m,\epsilon) \mbox{ is at most }\propto -\epsilon^{-u/3} \mbox{ for }0<u<3/2.\end{equation}

As a consequence, we obtain that 
\begin{equation}\log \rho(m,\epsilon) \propto  -\epsilon^{-u/3} \mbox{ for }0<u<3/2.\end{equation}

\subsubsection{Asymptotics of $\rho(m,\epsilon)$ with $m \propto \epsilon^{-u}$, $u\ge 3/2$, and $\rho(\infty,\epsilon)$}\label{s:u-critique}

The above discussion dealt only with rough order of magnitudes (denoted by the $\propto$ symbol), but a more precise analysis is needed to study
the competition between positive and negative terms of similar orders of magnitude when $u=3/2$. 

For $\lambda>0$, let $m:=\floor{ \lambda \epsilon^{-3/2} }$ and 
$$f^{+}(\lambda):=\limsup_{\epsilon \to 0} - \epsilon^{1/2} \log \rho(m, \epsilon), \ f^{-}(\lambda):=\liminf_{\epsilon \to 0} - \epsilon^{1/2} \log \rho(m, \epsilon).$$
 
From the monotonicity property: 
$\rho(m,\epsilon) \geq \rho(m',\epsilon)$ when $m' \geq m$, we deduce that $\lambda \mapsto f^{+}(\lambda)$ and $\lambda \mapsto f^{-}(\lambda)$ are non-decreasing.
Similarly, from the monotonicity property $\rho(m,\epsilon) \leq \rho(m,\epsilon')$ when $\epsilon \leq \epsilon'$, we deduce that
 $\lambda \mapsto \lambda^{-1/3} f^{+}(\lambda)$ and $\lambda \mapsto \lambda^{-1/3} f^{-}(\lambda)$ are non-increasing.

In particular, the fact that there exists some $\lambda$ for which $f^{+}(\lambda)$ is finite (resp. positive) implies that $f^{+}(\lambda)$ is finite (resp. positive) for all $\lambda>0$. 
The same property holds for $f^{-}$.

Now, rework the bounds in the previous section, replacing the $\epsilon^{-u/3}(=\epsilon^{-1/2}$ since $u=3/2$) terms 
in the definition of $\Xi_{m}$ and $\Delta_{m}$, by 
$\lambda \epsilon^{-1/2}$ (more precision than only the order of magnitude of terms is needed in order to deal with the case $u=3/2$).
 Consider the analog of the bound~(\ref{e:borne-partie-exp}) in the present context, and observe that, for small $\epsilon$, $\epsilon m \sim \lambda \epsilon^{-1/2}$. 
 The Brownian scaling bound then yields the existence of a constant $c>0$ such that,
  for small $\epsilon$, 
  $$\log \tilde{P}\left[\mbox{ $ \forall i \in \ig  0,m\id$, $v(p) i - \epsilon i \leq S_{i} \leq v(p) i + \lambda \epsilon^{-1/2}$ }\right] 
  \leq -c \frac{m}{(\lambda \epsilon^{-1/2})^{2}} \sim -c \lambda^{-1} \epsilon^{-1/2}.$$ 
  For small enough $\lambda$, this term dominates the $t^{*} \epsilon m \sim t^{*} \lambda \epsilon^{-1/2}$ term in the exponential, so that $f^{-}(\lambda)>0$. We deduce that
   $f^{-}(\lambda)>0$ for all values of $\lambda>0$. 
    On the other hand, it is  straightforward to adapt the estimates in the previous section to show that $f^{+}(\lambda)<+\infty$ for all $\lambda>0$. 
We can thus conclude that, when $m \propto \epsilon^{-3/2}$,
\begin{equation}\label{e:bientot-la-fin}\log \rho(m,\epsilon) \propto -\epsilon^{-1/2}.\end{equation}

Now, by Proposition~\ref{p:comparaison}, the asymptotic scaling
$\log \rho(\infty,\epsilon) \propto -\epsilon^{-1/2}$ is a consequence of~(\ref{e:bientot-la-fin}),
provided that~(\ref{e:assez-grand}) is satisfied for $u=3/2$ and, at least, large enough $\lambda$. 
Observe that, on the one hand, $q \sim \frac{\alpha \epsilon m}{v(p)-R} \sim \frac{\alpha \lambda \epsilon^{-1/2}}{v(p) -R }$. 
On the other hand, $ \log \rho(m,(1-\alpha)\epsilon) \gtrsim - f^{+}(\lambda (1-\alpha)^{3/2})   ((1-\alpha) \epsilon)^{-1/2}$. 
The fact that $\lambda \mapsto \lambda^{-1/3} f^{+}(\lambda)$ is non-increasing and thus bounded above for large $\lambda$, 
implies that $\phi^{q} \rho(m,(1-\alpha)\epsilon) >> 1$ for large enough $\lambda$, so that~(\ref{e:assez-grand}) is indeed satisfied.

\subsection{Deducing the Brunet-Derrida behavior}

Broadly speaking, our proof of the Brunet-Derrida behavior of branching-selection systems is based on the fact that there is a loose equivalence between the following two properties:
 \begin{equation}\label{e:borne-vitesse-vague-2} \BRW_1,\ldots,\BRW_N \mbox{ do not survive killing  below a line of slope } v-\epsilon,\end{equation}
and
\begin{equation}\label{e:borne-vitesse-vague} v_N < v-\epsilon. \end{equation}
If one accepts this premise, it is then natural to expect the actual velocity shift $\epsilon_N := v(p) - v_N$ to satisfy 
\begin{equation}\label{e:ordre-epsilon}\rho(\infty,\epsilon_N) \propto 1/N.\end{equation}
Indeed, since $\BRW_1,\ldots,\BRW_N$ are independent, $\rho(\infty,\epsilon_N)>>1/N$ would imply that, with probability close to one,  at least one of the $\BRW_i$s survives killing, while  $\rho(\infty,\epsilon_N)<<1/N$ would imply that, with probability close to one, none of the $\BRW_i$s survives.
Using the asymptotics $$\log \rho(\infty, \epsilon) \sim -\chi(p)^{1/2} \epsilon^{-1/2},$$
it is then easily checked that~(\ref{e:ordre-epsilon}) imposes the precise asymptotic behavior
$$\epsilon_N \sim \chi(p) (\log N)^{-2}.$$

To give an intuition of why~(\ref{e:borne-vitesse-vague-2}) and~(\ref{e:borne-vitesse-vague}) should be related, remember the coupling between the branching-selection particle system and $ \BRW_1,\ldots,\BRW_N$ described in Section~\ref{ss:cwafonbrw}. If~(\ref{e:borne-vitesse-vague-2}) holds, then, loosely speaking, $v(p)-\epsilon$ is above the sustainable growth speed for a branching system with only $N$ particles available, so the maximum of the branching-selection system should grow at a speed lower than $v(p)-\epsilon$. Conversely, if~(\ref{e:borne-vitesse-vague-2}) holds, the population in $ \BRW_1,\ldots,\BRW_N$ above the line with slope $v(p)-\epsilon$ quickly exceeds $N$ since the existing surviving particles quickly yield many surviving descendants. As a consequence, the threshold in the selection steps of the branching-selection system has to be above the line with slope $v(p)-\epsilon$.

A rigorous formulation  of the preceding arguments is  precisely what we do in Sections \ref{s:lower} and \ref{s:upper}.

It should be noted that the key time scale over which we have to control the particle system to prove both the upper and the lower bound is $\epsilon^{-3/2}$, 
with $\epsilon$ satisfying~(\ref{e:ordre-epsilon}), that is, a time scale of order $(\log N)^{3}$. 
This is the same order of magnitude as the one observed for coalescence times of the genealogical process underlying the branching-selection particle system 
(this question is investigated empirically and with heuristic arguments in e.g.~\cite{BruDerMueMun}). Understanding more precisely the role of this time-scale for the dynamics of the particle system certainly deserves more investigation.

\bibliographystyle{plain}
\bibliography{BD-particle}

\begin{thebibliography}{10}

\bibitem{AthNey}
K.~B. Athreya and P.~E. Ney.
\newblock {\em Branching processes}.
\newblock Dover Publications Inc., Mineola, NY, 2004.
\newblock Reprint of the 1972 original [Springer, New York; MR0373040].

\bibitem{BenDep2}
R.~Benguria and M.~C. Depassier.
\newblock On the speed of pulled fronts with a cutoff.
\newblock {\em Phys. Rev. E}, 75(5), 2007.

\bibitem{BenDep1}
R.~Benguria, M.~C. Depassier, and M.~Loss.
\newblock Validity of the {B}runet-{D}errida formula for the speed of pulled
  fronts with a cutoff.
\newblock {\em arXiv:0706.3671}, 2007.

\bibitem{Ber}
J.~B\'erard.
\newblock An example of {B}runet-{D}errida behavior for a branching-selection
  particle system on {Z}.
\newblock {\em arXiv:0810.5567}, 2008.

\bibitem{BruDerMueMun}
{\'E}.~Brunet, B.~Derrida, A.~H. Mueller, and S.~Munier.
\newblock Effect of selection on ancestry: an exactly soluble case and its
  phenomenological generalization.
\newblock {\em Phys. Rev. E (3)}, 76(4):041104, 20, 2007.

\bibitem{BruDer1}
Eric Brunet and Bernard Derrida.
\newblock Shift in the velocity of a front due to a cutoff.
\newblock {\em Phys. Rev. E (3)}, 56(3, part A):2597--2604, 1997.

\bibitem{BruDer2}
{\'E}ric Brunet and Bernard Derrida.
\newblock Microscopic models of traveling wave equations.
\newblock {\em Computer Physics Communications}, 121-122:376--381, 1999.

\bibitem{BruDer3}
{\'E}ric Brunet and Bernard Derrida.
\newblock Effect of microscopic noise on front propagation.
\newblock {\em J. Statist. Phys.}, 103(1-2):269--282, 2001.

\bibitem{ConDoe}
Joseph~G. Conlon and Charles~R. Doering.
\newblock On travelling waves for the stochastic
  {F}isher-{K}olmogorov-{P}etrovsky-{P}iscunov equation.
\newblock {\em J. Stat. Phys.}, 120(3-4):421--477, 2005.

\bibitem{DerSim1}
B.~Derrida and D.~Simon.
\newblock The survival probability of a branching random walk in presence of an
  absorbing wall.
\newblock {\em Europhys. Lett. EPL}, 78(6):Art. 60006, 6, 2007.

\bibitem{DumPopKap}
Freddy Dumortier, Nikola Popovi{\'c}, and Tasso~J. Kaper.
\newblock The critical wave speed for the
  {F}isher-{K}olmogorov-{P}etrowskii-{P}iscounov equation with cut-off.
\newblock {\em Nonlinearity}, 20(4):855--877, 2007.

\bibitem{Dur}
Richard Durrett.
\newblock {\em Probability: theory and examples}.
\newblock Duxbury Press, Belmont, CA, second edition, 1996.

\bibitem{GanHuShi}
N.~Gantert, Yueyun Hu, and Zhan Shi.
\newblock Asymptotics for the survival probability in a supercritical branching
  random walk.
\newblock {\em arXiv:0811.0262}, 2008.

\bibitem{MueMytQua}
C.~Mueller, L.~Mytnik, and J.~Quastel.
\newblock Small noise asymptotics of traveling waves.
\newblock {\em Markov Process. Related Fields}, 14, 2008.

\bibitem{MueMytQua2}
C.~Mueller, L.~Mytnik, and J.~Quastel.
\newblock Effect of noise on front propagation in reaction-diffusion equations
  of {KPP} type.
\newblock {\em arXiv:0902.3423}, 2009.

\bibitem{Pem}
R.~Pemantle.
\newblock Search cost for a nearly optimal path in a binary tree.
\newblock {\em arXiv:math/0701741}, 2007.

\bibitem{DerSim2}
Damien Simon and Bernard Derrida.
\newblock Quasi-stationary regime of a branching random walk in presence of an
  absorbing wall.
\newblock {\em J. Stat. Phys.}, 131(2):203--233, 2008.

\end{thebibliography}

\end{document}